\documentclass[a4paper,11pt]{amsart}
\usepackage[english]{babel}
\usepackage{amssymb}
\usepackage{amscd}
\usepackage{mathtools}
\usepackage[paper=a4paper,left=3cm,right=3cm,top=2cm,bottom=2.0cm]{geometry}
\usepackage{enumerate}
\usepackage{hyperref}
\usepackage{tikz}
\usepackage{tikz-cd}

\tikzset{
    labl/.style={anchor=north, rotate=90, inner sep=.5mm}
}

\makeatletter
\@namedef{subjclassname@2020}{\textup{2020} Mathematics Subject Classification}
\makeatother

\newtheorem{thm}{Theorem}[section]
\newtheorem*{thm-non}{Theorem}

\newtheorem{lem}[thm]{Lemma}
\newtheorem{prop}[thm]{Proposition}
\newtheorem{cor}[thm]{Corollary}
\theoremstyle{definition}
\newtheorem{defi}[thm]{Definition}
\newtheorem{rem}[thm]{Remark}
\newtheorem{exam}[thm]{Example}

\DeclareMathOperator\Db{D^{b}}
\newcommand\Fo{F^{[n]}}
\newcommand\Fn{F_{[n]}}
\newcommand\Gn{G_{[n]}}
\newcommand\Hn{H_{[n]}}
\newcommand\En{E^{[n]}}
\newcommand\Fmn{F^{[n]}}
\newcommand\Sn{S_{[n]}}
\newcommand\In{\iota^{[n]}}
\newcommand\It{\iota^{\times n}}
\DeclareMathOperator\OO{\mathcal{O}}
\DeclareMathOperator\ZZ{\mathcal{Z}}
\newcommand\Xn{X^{[n]}}
\DeclareMathOperator\Hom{Hom}
\DeclareMathOperator\End{End}
\DeclareMathOperator\Ext{Ext}
\DeclareMathOperator\Coh{Coh}

\DeclareMathOperator\M{M}

\DeclareMathOperator\Fix{Fix}

\DeclareMathOperator\Pic{Pic}
\DeclareMathOperator\NS{NS}

\DeclareMathOperator\supp{supp}

\DeclareMathOperator\Amp{Amp}

\begin{document}

\title{Descent of tautological sheaves from Hilbert schemes to Enriques manifolds}

\author{Fabian Reede}
\address{Institut f\"ur Algebraische Geometrie, Leibniz Universit\"at Hannover, Welfengarten 1, 30167 Hannover, Germany}
\email{reede@math.uni-hannover.de}


\begin{abstract}
Let $X$ be a K3 surface which doubly covers an Enriques surface $S$. If $n\in\mathbb{N}$ is an odd number, then the Hilbert scheme of $n$-points $\Xn$ admits a natural quotient $\Sn$. This quotient is an Enriques manifold in the sense of Oguiso and Schröer. In this paper we construct slope stable sheaves on $\Sn$ and study some of their properties.
\end{abstract}

\maketitle

\vspace{-.8cm}
\section*{Introduction}
In 1896 Federigo Enriques gave examples of smooth projective surfaces with irregularity $q=0$ and geometric genus $p_g=0$ which are not rational. Therefore these surfaces were counterexamples to a conjecture by Max Noether, which stated that surfaces with $q=p_g=0$ are rational. Nowadays such a surface is called an Enriques surface.

The canonical bundle $\omega_S$ of an Enriques surface $S$ has order two in the Picard group of $S$. The induced double cover turns out to be a K3 surface (a two dimensional hyperk\"ahler manifold), hence it is the universal cover of $S$. On the other hand, every K3 surface $X$ which admits a fixed point free involution doubly covers an Enriques surface $S$.

Mimicking this correspondence Oguiso and Schr\"oer defined higher dimensional analogues of Enriques surfaces, the so called Enriques manifolds in \cite{ogu}. To be precise a connected complex manifold that is not simply connected and whose universal cover is a hyperk\"aler manifold is called an Enriques manifold.

The following class of examples is of interest to us in this work: take an odd natural number $n\in \mathbb{N}$ and an Enriques surface $S$. We have the induced K3 surface $X$ with a fixed point free involution $\iota$ such that $S=X/\iota$. Since $n$ is odd we get an induced fixed point free involution $\In$ on the Hilbert scheme of $n$-points $\Xn$. The quotient of $\Xn$ by the involution $\In$ is an Enriques manifold $\Sn$ of dimension $2n$. We have an \'etale Galois cover $\rho: \Xn \rightarrow \Sn$.

In this article we construct and study stable sheaves on Enriques manifolds of type $\Sn$. The main idea is to start with slope stable sheaves on $\Xn$ and check if they descend to $\Sn$. Known examples of stable sheaves on $\Xn$ are given by the tautological bundles $\En$ associated to slope stable locally free sheaves $E$ on $X$.

For example, we prove that $\En$ descends to $\Sn$ if and only if $E$ descends to $S$. If $\En$ descends we have $\En\cong \rho^{*}\Fn$ for some locally free sheaf $\Fn$ on $\Sn$. We then show that it is possible to find an ample divisor $D\in \Amp(\Sn)$ such that $\Fn$ is slope stable with respect to $D$. Finally using results from Kim \cite{kim} and Yoshioka \cite{yoshi}, we are able to prove that, given certain conditions are satisfied, we have in fact a morphism
\begin{equation*}
(-)_{[n]}: \mathrm{M}_{S,d}(v,L) \rightarrow \mathcal{M}_{\Sn,D}(v_{[n]}),\,\,\,F\mapsto \Fn
\end{equation*}
between a moduli spaces of stable sheaves on $S$ and moduli space of stable sheaves on $\Sn$. This morphism identifies the former moduli space as a smooth connected component in the latter.

This paper consists of four sections. In Section \ref{sec1} we generalize some results concerning tautological bundles on Hilbert schemes of points. Section \ref{sec2} contains results about the descent of tautological sheaves from $\Xn$ to $\Sn$. We compute certain $\Ext$-spaces in Section \ref{sec3}. In the final Section \ref{sec4} we study the stability of sheaves on Enriques manilfolds of type $\Sn$.

\section{Stability of tautological sheaves on Hilbert schemes of points}\label{sec1}

Let $X$ be a smooth projective surface. The Hilbert scheme $\Xn:=\mathrm{Hilb}^n(X)$ classifies length $n$ subschemes in $X$, that is
\begin{equation*}
\Xn=\left\lbrace [Z]\,\lvert\, Z\subset X,\,\, \dim(Z)=0\,\,\text{and}\,\, \dim(\text{H}^0(Z,\OO_Z))=n \right\rbrace.
\end{equation*}
In fact $\Xn$ is smooth itself and has dimension $2n$, see \cite[Theorem 2.4]{fog}. Moreover $\Xn$ is a fine moduli space for the classification of length $n$ subschemes and comes with the universal length $n$ subscheme
\begin{equation*}
\mathcal{Z}=\left\lbrace (x,[Z])\in X\times\Xn\,\lvert\, x\in \supp(Z) \right\rbrace \subset X\times \Xn.
\end{equation*}
The universal subscheme $\mathcal{Z}$ comes with two projections $p: \mathcal{Z}\rightarrow \Xn$ and $q: \mathcal{Z}\rightarrow X$. Note that the morphism $p$ is finite and flat of degree $n$.

To any locally free sheaf $E$ of rank $r$ on $X$ one can associate the so called tautological vector bundle $\En$ on $\Xn$ via
\begin{equation*}
\En:=p_{*}q^{*}E.
\end{equation*}
As $p$ is finite and flat of degree $n$ the sheaf $\En$ is indeed locally free and has rank $nr$. The fiber at $[Z]\in \Xn$ can be computed to be
\begin{equation*}
\En\otimes \OO_{[Z]}\cong \mathrm{H}^0(Z,E_{\lvert Z}).
\end{equation*}

\begin{rem}
Note that the definition of $\En$ makes sense for $E$ a coherent sheaf on $X$ or even a complex $E\in \Db(X)$ in the derived category of $X$, see \cite[Definition 2.4]{krug4}.
\end{rem}

In \cite[Theorem 1.4, Theorem 4.9]{staple} Stapleton proves that if $h\in\Amp(X)$ is an ample divisor on $X$ and $E\not\cong \OO_X$ is a slope stable (with respect to $h$) locally free sheaf, then there is $H\in \Amp(\Xn)$ such that the associated tautological bundle $\En$ is slope stable with respect to $H$ on $\Xn$.

In fact Stapleton's result remains true, if we drop the locally free condition and allow for torsion free sheaves, see for example \cite[Proposition 2.4]{wandel} for a first step toward the following observation:
\begin{lem}
Assume $E$ is torsion free and slope stable with respect to $h\in \Amp(X)$ such that its double dual satisfies $E^{**}\neq \OO_X$, then the associated tautological sheaf $\En$ is slope stable with respect to some $H\in \Amp(\Xn)$.
\end{lem} 

\begin{proof}
Since $X$ is a smooth projective surface and $E$ is torsion free we can canonically embed $E$ into its double dual. This gives an exact sequence
	\begin{equation}\label{bidual}
		\begin{tikzcd}
			0 \arrow[r] & E \arrow[r] & E^{**} \arrow[r] & Q \arrow[r] & 0.
		\end{tikzcd}
	\end{equation}	 
Here $E^{**}$ is locally free and also slope stable with respect to $h$. Furthermore $Q$ has support of codimension two.

By \cite[Corollary 6]{scala} the functor $\left(-\right)^{[n]}:\Coh(X)\rightarrow\Coh(\Xn)$ is exact. So we get an exact sequence on $\Xn$
 	\begin{equation*}
 		\begin{tikzcd}
 			0 \arrow[r] & \En \arrow[r] & \left( E^{**}\right)^{[n]} \arrow[r] & Q^{[n]} \arrow[r] & 0.
 		\end{tikzcd}
 	\end{equation*}	 
By our assumptions $(E^{**})^{[n]}$ is slope stable with respect to some $H\in \Amp(\Xn)$. But $Q^{[n]}$ has support of codimension two in $\Xn$ so that $\En$ is isomorphic to $(E^{**})^{[n]}$ in codimension one and thus must be also be slope stable with respect to $H$.
\end{proof}

The previous lemma shows that for every slope stable $E$ with $E^{**}\not\cong \OO_X$ there is $H\in \Amp(\Xn)$ such that the tautological sheaf $\En$ is slope stable with respect to $H$. Since $E$ belongs to some moduli space $\mathrm{M}_{X,h}(r,\operatorname{c}_1,\operatorname{c}_2)$, one may ask how $H$ varies if $E$ varies in its moduli. We can answer this question in the case that all sheaves classified by $\mathrm{M}_{X,h}(r,\operatorname{c}_1,\operatorname{c}_2)$ are locally free.

\begin{prop}\label{oneh}
If $(r,\operatorname{c}_1,\operatorname{c}_2)\neq(1,0,0)$ is chosen such that for every $[E]\in \mathrm{M}_{X,h}(r,\operatorname{c}_1,\operatorname{c}_2)$ the sheaf $E$ is slope stable and locally free, then there is $H\in \Amp(\Xn)$ such that $\En$ is slope stable with respect to $H$ for all $[E]\in \mathrm{M}_{X,h}(r,\operatorname{c}_1,\operatorname{c}_2)$.
\end{prop}

\begin{proof}
By a result of Stapleton, see \cite[Theorem 1.4]{staple}, we know that for $[E]\in \mathrm{M}_{X,h}(r,\operatorname{c}_1,\operatorname{c}_2)$ the locally free sheaf $\En$ is slope stable with respect to the induced nef divisor $h_n\in \NS(\Xn)$. It is also well known that the Hilbert - Chow morphism $\mathsf{HC}: \Xn\rightarrow X^{(n)}$ is semi-small and that $q: \mathcal{Z}\rightarrow X$ is flat, see \cite[Theorem 2.1]{krug2}. 

The proof is now exactly the same as for tautological bundles on the generalized Kummer variety $\mathsf{Kum}_n(A)$ associated to an abelian surface $A$, see \cite[Proposition 2.9]{rz}.
\end{proof}

\begin{rem}
    The condition that all sheaves in $\mathrm{M}_{X,h}(r,\operatorname{c}_1,\operatorname{c}_2)$ are slope stable can be achieved (for example) in the following two different ways: the first is by a special choice of the numerical invariants, see \cite[Lemma 1.2.14]{lehn}. The second way is by choosing a special ample class $h$, see \cite[Theorem 4.C.3]{lehn}. 
    
    To find a moduli space such that all sheaves are locally free, one can do the following: if the tuple $(r,\operatorname{c}_1)$ is fixed, then by Bogomolov's inequality the second Chern class is bounded from below, see \cite[Theorem 3.4.1]{lehn}. Choose the minimal $\operatorname{c}_2$, then every sheaf in $\mathrm{M}_{X,h}(r,\operatorname{c}_1,\operatorname{c}_2)$ is locally free. Indeed, if such an $E$ is not locally free, then $E^{**}$ is locally free, stable with respect to $h$ and has the same tuple $(r,\operatorname{c}_1)$, but it has smaller $\operatorname{c}_2$ by exact sequence \eqref{bidual} as $\operatorname{c}_2(Q)<0$, contradicting minimality. See also \cite[Remark 6.1.9]{lehn} for a similar argument.
\end{rem}

Now let $X$ be a K3 surface. Denote the Mukai vectors of $E$ and $\En$ by $v$ respectively $v^{[n]}\in \mathrm{H}^{*}(\Xn,\mathbb{Q})$. If $\En$ is slope stable, then it belongs to the moduli space $\mathcal{M}_{\Xn,H}(v^{[n]})$ of semistable sheaves on $\Xn$ with Mukai vector $v^{[n]}$. In fact we can generalize \cite[Corollary 4.6]{wandel} to get the following

\begin{thm}\label{compo}
If $v\neq v(\OO_X)$ is a Mukai vector such that for every $[E]\in \mathrm{M}_{X,h}(v)$ the sheaf $E$ is slope stable, locally free and $h^i(X,E)=0$ for $i=1,2$, then the functor $(-)^{[n]}$ induces a morphism
\begin{equation*}
(-)^{[n]}: \M_{X,h}(v) \rightarrow \mathcal{M}_{\Xn,H}(v^{[n]}),\,\,\, [E]\mapsto [\En]
\end{equation*}
which identifies $\mathrm{M}_{X,h}(v)$ with a smooth connected component of $\mathcal{M}_{\Xn,H}(v^{[n]})$.
\end{thm}

\begin{proof}
First note that the map $[E] \mapsto [\En]$ is indeed a regular morphism, see for example \cite[Proposition 2.1]{krug5}. Furthermore this morphism is injective on closed points, which follows immediately from \cite[Theorem 1.1]{bis} (see also \cite[Theorem 1.2]{krug2} for a generalization).

By \cite[Corollary 4.2 (11)]{krug4} we find
\begin{equation*}
\Ext^1_{\Xn}(\En,\Fo)\cong\Ext^1_X(E,F)
\end{equation*}
since $h^0(X,E^{\vee})=h^2(X,E)=0$ as well as $h^1(X,E^{\vee})=h^1(X,E)=0$. For $E=F$ this isomorphisms translates to
\begin{equation*}
\dim(T_{[\En]}\mathcal{M}_{\Xn,H}(v^{[n]}))=\dim(T_{[E]}\mathrm{M}_{X,h}(v)).
\end{equation*}
These two facts imply that we can identify $\mathrm{M}_{X,h}(v)$ with a smooth connected component in $\mathcal{M}_{\Xn,H}(v^{[n]})$.
\end{proof}

\section{Descent of tautological sheaves  to Enriques manifolds}\label{sec2}
Let $G$ be a finite group. Consider an \'etale Galois cover $\varphi: Y \rightarrow Z$ with Galois group $G$, that is there is a free $G$-action on $Y$ such that $Z=Y/G$ and $\varphi$ is the quotient map.
In this situation there is an equivalence between the categories $\Coh(Z)$ of coherent sheaves on $Z$ and $\Coh^G(Y)$ of $G$-equivariant coherent sheaves on $Y$ given by the functors
\begin{align*}
&\varphi^{*}: \Coh(Z) \rightarrow \Coh^G(Y),\,\,\, E \mapsto \varphi^{*}E\,\,\text{and}\\
&\varphi^G_{*}: \Coh^G(Y) \rightarrow \Coh(Z),\,\,\, F\mapsto \left( \varphi_{*}(F)\right)^G 
\end{align*}
We say that a coherent sheaf $E$ on $Y$ descends to $Z$ if $E$ is in the image of $\varphi^{*}$, that is there is a coherent sheaf $F$ on $Z$ together with an isomorphism $E\cong\varphi^{*}(F)$.

A coherent sheaf $E$ on $X$ is said to be $G$-invariant, if there are isomorphisms $E\cong g^{*}E$ for every $g\in G$. A $G$-equivariant coherent sheaf is $G$-invariant, but the converse is not true. For our purposes the following will suffice, see \cite[Lemma 1]{ploog}:

\begin{prop}\label{desc}
Assume that $G$ is a cyclic group and $E$ is a simple $G$-invariant coherent sheaf on $Y$, then $E$ descends to $Z$.
\end{prop}

\begin{rem}
Recall that if $(X,\iota)$ is a pair consisting of a K3 surface and a fixed point free involution, then $G=\left\langle \iota \right\rangle \cong \mathbb{Z}/2\mathbb{Z}$ acts freely on $X$ and the quotient $S$ is an Enriques surface. The morphism $\pi: X\rightarrow S$ is an \'etale $\mathbb{Z}/2\mathbb{Z}$-Galois cover. 

On the other hand if $S$ is an Enriques surface, then its canonical bundle $\omega_S$ is 2-torsion. One can consider the induced canonical cover $\phi: \tilde{S}:=\operatorname{Spec}(\OO_S\oplus\omega_S)\rightarrow S$. The morphism $\phi$ is an \'etale $\mathbb{Z}/2\mathbb{Z}$-Galois cover and $\tilde{S}$ is a K3 surface with fixed point free involution, the covering involution of $\phi$. Furthermore $\phi_{*}\OO_{\tilde{S}}\cong\OO_S\oplus\omega_S$.
\end{rem}

In \cite{ogu} Oguiso and Schr\"oer generalized the notion of an Enriques surface to that of an Enriques manifold by mimicking the above constructions:
\begin{defi}
A manifold $Y$ is called an Enriques manifold if it is a connected complex manifold that is not simply connected and whose universal cover is a hyperk\"ahler manifold.
\end{defi}

\begin{rem}
    In \cite{sarti2} the authors also gave a definition of higher dimensional Enriques varieties, which slightly differs from the one of Enriques manifolds in \cite{ogu}.
\end{rem}

\begin{rem}
An Enriques manifold is of even dimension, say $\dim(Y)=2n$. The fundamental group $\pi_1(Y)$ is finite of order $d$ with $d\,\lvert\,n+1$. This number $d$ is called the index of $Y$. In addition $Y$ is projective and the canonical bundle $\omega_Y$ has finite order $d$ and generates the torsion group of $\Pic(Y)$, see \cite[Section 2]{ogu}.
\end{rem}

We will work with the following class of Enriques manifolds, see \cite[Proposition 4.1]{ogu}: 

\begin{exam}
Let $(X,\iota)$ be a pair consisting of a K3 surface together with a fixed point free involution $\iota$ on $X$. Then $X$ covers the Enriques surface $S=X/\iota$. If $n\in\mathbb{N}$ is odd, then $(X,\iota)$ induces the pair $(\Xn,\In)$ of the Hilbert scheme of $n$-points on $X$ and the induced fixed point free involution $\In$ on $\Xn$. Thus $G=\left\langle \In \right\rangle \cong \mathbb{Z}/2\mathbb{Z}$ acts freely on $\Xn$ and the quotient $\Sn$ is an Enriques manifold with index $d=2$ coming with an \'etale $\mathbb{Z}/2\mathbb{Z}$-cover $\rho: \Xn\rightarrow \Sn$. 
\end{exam}

We want to study the descent of sheaves from $X$ to $S$ respectively from $\Xn$ to $\Sn$. To do this we need the following lemma:
\begin{lem}\label{commute}
There is an isomorphism of functors from $\Coh(X)$ to $\Coh(\Xn)$:
\begin{equation*}
(\In)^{*}\left( (-)^{[n]}\right)  \cong (\iota^{*}(-))^{[n]}.
\end{equation*}
\end{lem}
\begin{proof}
Recall that $(-)^{[n]}=\mathrm{FM}_{\OO_{\ZZ}}(-)$ can be written as the Fourier -- Mukai transform with kernel the structure sheaf of universal family $\ZZ$ in $X\times \Xn$, see for example \cite[Section 2.3]{krug2}. Define a group isomorphism 
\begin{equation*}
\mu: \left\langle \iota \right\rangle  \rightarrow \left\langle \In \right\rangle ,\,\,\,\iota \mapsto \In
\end{equation*}
and note that this is a so-called $c$-isomorphism, see \cite[Definitions 3.1 and 3.3]{krug}. By the definition of the universal family we see that there is an isomorphism
\begin{equation*}
(\iota\times \mu(\iota))^{*}\OO_{\ZZ} = (\iota\times\In)^{*}\OO_{\ZZ} \cong \OO_{\ZZ}.
\end{equation*}
Thus $\OO_{\ZZ}$ is $\mu$-invariant, see \cite[Definition 3.4]{krug}, which implies
\begin{equation*}
(\In)^{*}\left( \mathrm{FM}_{\OO_{\ZZ}}(-)\right) \cong \mathrm{FM}_{\OO_{\ZZ}}(\iota^{*}(-))
\end{equation*}
by \cite[Lemma 3.6 (iii)]{krug}.
\end{proof}

We can now prove the main result of this section:
\begin{thm}\label{desce}
Assume $(X,\iota)$ is a K3 surface together with a fixed point free involution and let $n\in\mathbb{N}$ be an odd number. If a torsion free sheaf $E$ on $X$ is simple, then the associated tautological sheaf $\En$ on $\Xn$ descends to $\Sn$ if and only if $E$ descends to $S$.  
\end{thm}

\begin{proof}
First we note that if $E$ is simple then $\En$ is also simple. Indeed  by \cite[Corollary 4.2 (11)]{krug4} there is an isomorphism
\begin{equation*}
    \End_{\Xn}(\En)\cong \End_X(E)\oplus \mathrm{H}^0(X,E^{*})\otimes \mathrm{H}^0(X,E).
\end{equation*}
Since $E$ is simple the second summand must vanish, since otherwise $E$ would have an endomorphism, which has image of rank one and thus is no homothety. 

Proposition \ref{desc} shows
\begin{equation*}
\En\,\,\text{decends to}\,\, \Sn \Leftrightarrow (\In)^{*}\En\cong \En.
\end{equation*}
By Lemma \ref{commute} we get
\begin{equation*}
(\In)^{*}\En\cong \En \Leftrightarrow (\iota^{*}E)^{[n]}\cong \En.
\end{equation*}
But \cite[Theorem 1.2]{krug2} shows
\begin{equation*}
(\iota^{*}E)^{[n]}\cong \En \Leftrightarrow \iota^{*}E\cong E.
\end{equation*}
Thus $\En$ descends to $\Sn$ if and only if $E$ descends to $S$.
\end{proof}

The theorem shows that given a simple $\iota$-invariant torsion free sheaf $E$ on $X$ then there is $F\in \Coh(S)$ and $G\in \Coh(\Sn)$ such that
\begin{equation*}
E\cong \pi^{*}F\,\,\,\text{as well as}\,\,\, \En\cong\rho^{*}G.
\end{equation*}
In fact, there is a close relationship between the sheaves $F$ and $G$: as $\OO_{\mathcal{Z}}$ is $\mu$-invariant on $X\times\Xn$, the structure sheaf $\OO_{\mathcal{Z}}$ is naturally $\mu$-linearizable on $\mathcal{Z}$, hence so is $\OO_{\mathcal{Z}}$ as a sheaf on $X\times\Xn$.

Therefore by \cite[Proposition 4.2]{krug} the functor 
$(-)^{[n]}$ descends to a functor 
\begin{equation*}
(-)_{[n]}: \Db(S) \rightarrow \Db(\Sn)
\end{equation*}
together with a commutative diagram
 	 	\begin{equation}\label{diag}
 		\begin{tikzcd}
 			\Db(S) \arrow{r}{\pi^{*}}[swap]{\cong}\arrow[d,"(-)_{[n]}"]\arrow[rr,bend left=20,"\pi^{*}"]& \mathrm{D}^b_{\iota}(X)\arrow[r,"\mathrm{For}"]\arrow[d,"(-)^{[n]}_{\mathbb{Z}/2\mathbb{Z}}"] & \Db(X)\arrow[d,"(-)^{[n]}"]\\
    \Db(\Sn) \arrow{r}{\rho^{*}}[swap]{\cong} \arrow[rr,bend right=20, "\rho^{*}"]& \mathrm{D}^b_{\In}(\Xn)\arrow[r,"\mathrm{For}"] & \Db(\Xn)
 		\end{tikzcd}
 	\end{equation}	
Here $\mathrm{For}$ is the functor forgetting the linearizations.

That is if we start with a simple sheaf $E$ on $X$, which descends to $S$ i.e. $E\cong \pi^{*}F$, then $\En$ descends to $\Sn$ with $\En\cong \rho^{*}\Fn$.

\begin{rem}\label{choices}
    As $\OO_{\ZZ}$ has two choices of a $\mu$-linearization (differing by the non-trivial character), there are actually two choices of the descent $(-)_{[n]}: \Db(S)\rightarrow \Db(\Sn)$ (differing by tensor product by $\omega_{\Sn}$).
\end{rem}

We end this section by giving a more explicit description of $(-)_{[n]}$ similar to $(-)^{[n]}$. For this recall that by \cite[2.4]{krug5} we have
\begin{equation*}
    (-)^{[n]}=\mathrm{FM}_{\OO_{\mathcal{Z}}}(-)={p_{\Xn}}_{*}(p_X^{*}(-)),
\end{equation*}
where $p_X: \mathcal{Z}\rightarrow X$ and $p_{\Xn}: \mathcal{Z}\rightarrow \Xn$ are the projections.

The group $G=\mathbb{Z}/2\mathbb{Z}$ acts freely on $X$ via $\iota$ with quotient $S$, freely on $\Xn$ via $\iota^{[n]}$ with quotient $\Sn$ and thus also freely on $X\times\Xn$ via $\iota\times\iota^{[n]}$. 
As the universal family $\mathcal{Z}\hookrightarrow X\times \Xn$ is $G$-invariant, we get a closed subvariety $\mathcal{Z}/G\hookrightarrow (X\times \Xn)/G$. Furthermore the projections $p_X$ and $p_{\Xn}$ are $G$-equivariant. By \cite[Lemma 2.3.3]{haut} we get cartesian squares
 	 	\begin{equation}\label{diag2}
 		\begin{tikzcd}
 			X \arrow[d,swap,"\pi"]& \mathcal{Z} \arrow[l,swap,"p_X"] \arrow[r,"p_{\Xn}"]\arrow[d,"\alpha"] & \Xn\arrow[d,"\rho"]\\
    S & \mathcal{Z}/G \arrow[l,swap,"p_S"] \arrow[r,"p_{\Sn}"] & \Sn
 		\end{tikzcd}
 	\end{equation}	

\begin{thm}\label{descr}
    The functor $(-)_{[n]}:\Db(S)\rightarrow \Db(\Sn)$ has the following description:
    \begin{equation*}
        (-)_{[n]}={p_{\Sn*}}(p_S^{*}(-)).
    \end{equation*}
\end{thm}

\begin{proof}
    From diagram \eqref{diag} we see that $\rho^{*}((-)_{[n]})=(\pi^{*}(-))^{[n]}$. Since $\rho_{*}(\rho^{*}(-))^G= \mathrm{id}$ we find
    \begin{align*}
        (-)_{[n]}&= \rho_{*}((\pi^{*}(-))^{[n]})^G = \rho_{*}({p_{\Xn}}_{*}(p_X^{*}(\pi^{*}(-))))^G\\
        &={p_{\Sn*}}(\alpha_{*}(\alpha^{*}(p_S^{*}(-))))^G={p_{\Sn*}}(\alpha_{*}(\alpha^{*}(p_S^{*}(-)))^G)\\
        &={p_{\Sn*}}(p_S^{*}(-)).
    \end{align*}
    Here we used the commutativity of diagram \eqref{diag2}, the fact that $G$ acts trivially on $\mathcal{Z}/G$ and $\Sn$ hence by \cite[Equation (5)]{krug4} we have $(-)^G{p_{\Sn*}}={p_{\Sn*}}(-)^G$ and the $G$-equivariant projection formula.
\end{proof}

\section{Computation of certain extension spaces}\label{sec3}
In \cite[Theorem 3.17]{krug6} Krug gave explicit formulas for homological invariants of tautological objects in $\Db(\Xn)$ in terms of those in $\Db(X)$, for example for $E,F\in \Db(X)$ there is an isomorphism of graded vector spaces:
\begin{align*}
\Ext^{*}_{\Xn}(\En,\Fmn)\cong &\Ext^{*}_X(E,F)\otimes S^{n-1}\mathrm{H}^{*}(X,\OO_X)\\
&\oplus \Ext^{*}_X(E,\OO_X)\otimes\Ext^{*}_X(\OO_X,F)\otimes S^{n-2}\mathrm{H}^{*}(X,\OO_X).
\end{align*}
See also \cite[Section 4]{krug4} for a considerably simplified proof of this formula.

In this section we want to find homological invariants of sheaves of the form $\Gn$ on $\Sn$ in terms of the sheaf $G$ on $S$. It is certainly possible to find a general formula similar to Krug's result, but to keep formulas and proofs short and readable and since it is enough for our purposes, we will restrict our attention to $\Hom$- and $\Ext^1$- spaces as well as sheaves without higher cohomology. We will use the notations and results from \cite{krug4}.

We start by studying how Krug's result behaves with respect to the group actions by $\mathbb{Z}/2\mathbb{Z}$ on $\Xn$ via $\In$ and on $X$ via $\iota$. We will denote the various versions of the group $G=\mathbb{Z}/2\mathbb{Z}$ in the following by their nontrivial element, that is by $\iota$ or $\In$ etc.

\begin{lem}\label{iotas}
Assume $(X,\iota)$ is a K3 surface together with a fixed point free involution. For $\iota$-equivariant coherent sheaves $E,F\in \Coh_{\iota}(X)$ there is an isomorphism of graded vector spaces:
\begin{align*}
\left( \Ext^{*}_{\Xn}(\En,\Fmn)\right)^{\In} \cong &\left( \Ext^{*}_X(E,F)\otimes S^{n-1}\mathrm{H}^{*}(X,\OO_X)\right)^{\iota} \\
&\oplus \left( \Ext^{*}_X(E,\OO_X)\otimes\Ext^{*}_X(\OO_X,F)\otimes S^{n-2}\mathrm{H}^{*}(X,\OO_X)\right)^{\iota}.
\end{align*}
\end{lem}

\begin{proof}
Note that on the right hand side of the formula we take invariants with respect to the actions induced by the linearizations of $E$, $F$ and $\OO_X$. On the left hand side we take invariants with respect to the \emph{induced} linearizations on $\En$ and $\Fo$. The existence of the induced linearizations follows from the right-hand side of diagram \eqref{diag}.

By \cite[Theorem 3.6]{krug4} there is an isomorphism of functors
\begin{equation}\label{isofunc}
\left(-\right)^{[n]}\cong\Psi\circ \mathsf{C},
\end{equation}
where $\mathsf{C}: \Coh(X) \rightarrow \Coh_{\mathfrak{S}_n}(X^n)$ is the exact functor with
\begin{equation*}
\mathsf{C}(E):=\mathsf{Ind}_{\mathfrak{S}_{n-1}}^{\mathfrak{S}_n}\mathrm{pr}_1^{*}E\cong \bigoplus\limits_{i=1}^n \mathrm{pr}_i^{*}E.
\end{equation*}

Furthermore $\Psi: \mathrm{D}^{\mathrm{b}}_{\mathfrak{S}_n}(X^n) \rightarrow \Db(\Xn)$ is the Fourier - Mukai transform with kernel the structure sheaf of the isospectral Hilbert scheme $I^nX$. Here the isosprectral Hilbert scheme is the reduced fiber product $I^nX:=(\Xn\times_{S^nX}X^n)_{\mathrm{red}}$ of the quotient map $\nu:X^n\rightarrow S^nX$ to the symmetric power and the Hilbert - Chow morphism $\mu: \Xn \rightarrow S^nX$. This Fourier - Mukai transform is an equivalence, see \cite[Proposition 2.8]{krug4} and satisfies 
\begin{equation}\label{psiinv}
\left( \In\right)^{*}\circ\Psi=\Psi\circ \left( \It\right)^{*}
\end{equation}
see for example \cite[Section 5.6]{krug}. Here $\It$ is the induced involution on $X^n$. 

We have the following chain of isomorphisms:
\begin{align*}
\left( \Ext^{*}_{\Xn}(\En,\Fmn)\right)^{\In} &\cong \left( \Ext^{*}_{\Xn}(\Psi(\mathsf{C}(E)),\Psi(\mathsf{C}(F)))\right)^{\In}\\
&\cong \left( \Ext^{*}_{X^n,\mathfrak{S}_n}(\mathsf{C}(E),\mathsf{C}(F))\right)^{\It}\\
&\cong \left( \Ext^{*}_{X^n,\mathfrak{S}_{n-1}}(\mathrm{pr}_1^{*}E,\mathrm{pr}_1^{*}F)\right)^{\It}\oplus \left( \Ext^{*}_{X^n,\mathfrak{S}_{n-2}}(\mathrm{pr}_1^{*}E,\mathrm{pr}_2^{*}F)\right)^{\It}
\end{align*}
Here the first isomorphism is \eqref{isofunc}. The second isomorphism uses that $\Psi$ is an equivalence and \eqref{psiinv}. The last isomorphism can be extracted from \cite[Proposition 4.1]{krug4}.

We look at the first summand, the second working similarly. First note that
\begin{equation*}
\mathrm{pr}_1^{*}E=E\,\boxtimes\,\OO_X\,\boxtimes\,\cdots\,\boxtimes\,\OO_X. 
\end{equation*}
Applying the K\"unneth formula shows
\begin{align*}
\Ext^{*}_{X^n,\mathfrak{S}_{n-1}}(\mathrm{pr}_1^{*}E,\mathrm{pr}_1^{*}F)&=\Ext^{*}_{X^n,\mathfrak{S}_{n-1}}(E\,\boxtimes\,\OO_X\,\boxtimes\,\cdots\,\boxtimes\,\OO_X,F\,\boxtimes\,\OO_X\,\boxtimes\,\cdots\,\boxtimes\,\OO_X)\\
&\cong \left( \Ext^{*}_X(E,F)\otimes \mathrm{H}^{*}(X,\OO_X)^{\otimes n-1}\right)^{\mathfrak{S_{n-1}}}
\end{align*}
But the group $\mathbb{Z}/2\mathbb{Z}$ acts on sheaves of the form $\mathrm{pr}_1^{*}E$ by definition of $\It$ as
\begin{equation*}
\left( \It\right)^{*} \mathrm{pr}_1^{*}E=\iota^{*}E\,\boxtimes\,\iota^{*}\OO_X\,\boxtimes\,\cdots\,\boxtimes\,\iota^{*}\OO_X
\end{equation*}
that is simply by the pullback via $\iota$ on each factor in the box product. Since the action of $\mathbb{Z}/2\mathbb{Z}$ via $\It$ and the $\mathfrak{S}_n$ action commute we finally see that:
\begin{equation*}
\left( \Ext^{*}_{X^n,\mathfrak{S}_{n-1}}(\mathrm{pr}_1^{*}E,\mathrm{pr}_1^{*}F)\right)^{\It}\cong \left(\Ext^{*}_X(E,F)\otimes S^{n-1}\mathrm{H}^{*}(X,\OO_X) \right)^{\iota}.\rlap{$\qquad \Box$}
\end{equation*}\phantom\qedhere
\end{proof}

\begin{thm}\label{homext}
Let $(X,\iota)$ be a K3 surface together with a fixed point free involution and let $n\in\mathbb{N}$ be an odd number. If $G,H\in \Coh(S)$ are such that $\pi^{*}G$ and $\pi^{*}H$ have no higher cohomology (here $S=X/\iota$ is the associated Enriques surface), then
\begin{equation*}
\Hom_{\Sn}(\Gn,\Hn)\cong \Hom_S(G,H)\,\,\,\text{and}\,\,\,\Ext^1_{\Sn}(\Gn,\Hn)\cong \Ext^1_S(G,H).
\end{equation*}
\end{thm}

\begin{proof}
Define $E:=\pi^{*}G$ and $F:=\pi^{*}H$. It follows from diagram \eqref{diag} that $\En\cong\rho^{*}\Gn$ and $\Fo\cong\rho^{*}\Hn$. We therefore have an isomorphism
\begin{equation*}
\Ext_{\Sn}^{*}(\Gn,\Hn) \cong\left( \Ext_{\Xn}^{*}(\rho^{*}\Gn,\rho^{*}\Hn)\right) ^{\In}\cong \left( \Ext_{\Xn}^{*}(\En,\Fo)\right) ^{\In}.
\end{equation*}
By Lemma \ref{iotas} the last space is isomorphic to 
\begin{equation}\label{suminv}
\left( \Ext^{*}_X(E,F)\otimes S^{n-1}\mathrm{H}^{*}(X,\OO_X)\right)^{\iota}
\oplus \left( \Ext^{*}_X(E,\OO_X)\otimes\mathrm{H}^{*}(X,F)\otimes S^{n-2}\mathrm{H}^{*}(X,\OO_X)\right)^{\iota}.
\end{equation}
We begin investigating the first summand. The natural $\mathbb{Z}/2$-linearization of $\OO_X$ induces an $\mathbb{Z}/2$-linearization on $\pi_{*}\OO_X\cong \OO_S\oplus\omega_S$ given by the generator of $\mathbb{Z}/2$ acting by $+1$ on $\OO_S$ and by $-1$ on $\omega_S$, see for example \cite[Remarks on p.72]{mum}. Hence $\iota$ acts as $+1$ on $\mathrm{H}^0(X,\OO_X)\cong \mathrm{H}^0(S,\OO_S)$ and by $-1$ on $\mathrm{H}^2(X,\OO_X)\cong \mathrm{H}^2(S,\omega_S)$. Furthermore, by the adjunction between $\pi^{*}$ and $\pi_{*}$ together with the projection formula, we get a splitting 
\begin{equation*}
\Ext^{*}_X(E,F) \cong \Ext^{*}_S(G,H)\oplus \Ext^{*}_S(G,H\otimes\omega_S).
\end{equation*}
where $\iota$ acts as $+1$ on the first summand and by $-1$ on the second summand.

Thus writing $\mathrm{H}^{*}(X,\OO_X)=\mathbb{C}[t]/(t^2)$ with $\deg(t)=2$ we get 
\begin{equation*}
S^{n-1}\mathrm{H}^{*}(X,\OO_X)=\mathbb{C}[t]/(t^n),\,\,\,\deg(t)=2
\end{equation*}
and $\iota$ acts as $+1$ on the constants and as $-1$ on $t$.

We can now compute the invariants and find
\begin{equation*}
\left( \Ext^{*}_S(G,H)\otimes\mathbb{C}[t]/(t^n)\right)^{\iota}= \Ext^{*}_S(G,H)\otimes\mathbb{C}[t^2]/(t^n),\,\,\,\deg(t)=2
\end{equation*}
as well as
\begin{equation*}
\left( \Ext^{*}_S(G,H\otimes\omega_S)\otimes\mathbb{C}[t]/(t^n)\right)^{\iota}= \Ext^{*}_S(G,H\otimes \omega_S)\otimes t\mathbb{C}[t^2]/(t^n),\,\,\,\deg(t)=2.
\end{equation*}
Looking at the components in degree zero and one sees
\begin{align*}
&\left( \left( \Ext^{*}_X(E,F)\otimes S^{n-1}\mathrm{H}^{*}(X,\OO_X)\right)^{\iota}\right) _{0}\cong \Hom_S(G,H)\,\,\,\text{as well as}\\
&\left( \left( \Ext^{*}_X(E,F)\otimes S^{n-1}\mathrm{H}^{*}(X,\OO_X)\right)^{\iota}\right) _{1}\cong \Ext^1_S(G,H).
\end{align*}

Next we study the second summand in \eqref{suminv}: since $E$ and $F$ have no higher cohomology we have
\begin{equation*}
\Ext^{*}_X(E,\OO_X)\otimes\mathrm{H}^{*}(X,F)\cong \Ext^2_X(E,\OO_X)\otimes \mathrm{H}^{0}(X,F)
\end{equation*}
which already lives in degree two. As we also have
\begin{equation*}
S^{n-2}\mathrm{H}^{*}(X,\OO_X)=\mathbb{C}[t]/(t^{n-1}),\,\,\,\deg(t)=2,
\end{equation*}
we see that the second summand in \eqref{suminv} can possibly have nontrivial components starting in degrees at least two. Especially for $k\in\left\lbrace 0,1 \right\rbrace$ we find
\begin{equation*}
\left( \left( \Ext^{*}_X(E,\OO_X)\otimes\mathrm{H}^{*}(X,E)\otimes S^{n-2}\mathrm{H}^{*}(X,\OO_X)\right)^{\iota}\right)_k=0.
\end{equation*}
Therefore we must have the desired isomorphisms
\begin{equation*}
\Hom_{\Sn}(\Gn,\Hn)\cong \Hom_S(G,H)\,\,\text{and}\,\,\Ext^1_{\Sn}(\Gn,\Hn)\cong \Ext^1_S(G,H).\rlap{$\qquad \Box$}
\end{equation*}\phantom\qedhere
\end{proof}

\section{Stable sheaves on Enriques manifolds}\label{sec4}
In this section we want to study the slope stability of sheaves of the form $\Fn$ on $\Sn$. For this we first recall the following fact:
let $\varphi: Y \rightarrow Z$ be an \'etale Galois cover with finite Galois group $G$ then there is the following relationship between slopes with respect to $h\in \Amp(Z)$:
\begin{equation}\label{slopes}
\mu_{\varphi^{*}h}(\varphi^{*}F) = \lvert G \rvert\,\mu_h(F).
\end{equation}
Using this fact we can prove the following lemma:

\begin{lem}\label{stability}
Let $E$ be a torsion free coherent sheaf on $Y$, slope stable with respect to $\varphi^{*}h$ for some $h\in \Amp(Z)$. If $E$ descends to $Z$, that is $E\cong\varphi^{*}F$, then $F$ is slope stable with respect to $h$.
\end{lem}
\begin{proof}
Let $H\subset F$ be a subsheaf of $F$. Then $\varphi^{*}H$ is a subsheaf of $\varphi^{*}F\cong E$. Since $E$ is slope stable with respect to $\varphi^{*}h$ we have 
\begin{equation*}
\mu_{\varphi^{*}h}(\varphi^{*}H) < \mu_{\varphi^{*}h}(E)=\mu_{\varphi^{*}h}(\varphi^{*}F)
\end{equation*}
which by \eqref{slopes} implies
\begin{equation*}
\mu_h(H)<\mu_h(F).
\end{equation*}
Hence $F$ is slope stable with respect to $h$.
\end{proof}

For the rest of this section we let $(X,\iota)$ be a K3 surface together with a fixed point free involution $\iota$. We denote the associated Enriques surface by $S$.

To prove the main theorem in this section we need the following isomorphism:
\begin{equation*}
\NS(\Xn)\cong\NS(X)_n\oplus \mathbb{Z}\delta.
\end{equation*}

\begin{rem}
   The summand $\NS(X)_n$ is constructed as follows: take $d\in \NS(X)$ and consider the element 
\begin{equation*}
    D^n:=\sum\limits_{i=1}^n \mathrm{pr}_i^{*}d\in \NS(X^n).
\end{equation*}
This element is $\mathfrak{S}_n$-invariant and thus descends to the symmetric product $S^nX$ by \cite[Lemma 6.1]{fog2}. More exactly, there is an element $D_{n}\in \NS(S^nX)$ such that $\nu^{*}D_{n}=D^n$ for the quotient map $\nu: X^n\rightarrow S^nX$. Then we define $d_n:=\mu^{*}D_n$, where $\mu:\Xn\rightarrow S^nX$ is the Hilbert - Chow morphism. 
\end{rem}

By \cite[Section 3]{sarti} the involution $\In$ acts on $\NS(X)_n$ via:
\begin{equation}\label{iondiv}
\left(\In \right)^{*}(d_n)=\left(\iota^{*}d \right)_n. 
\end{equation}

We are now ready to prove the main result of this section:

\begin{thm}\label{Fnstb}
 Assume $E\in \Coh(X)$ satisfies $E^{**}\not\cong \OO_X$, is torsion free and slope stable with respect to $h=\pi^{*}d$ for some $d\in \Amp(S)$. If $E$ descends to $S$, that is $E\cong \pi^{*}F$ for some $F\in \Coh(S)$, then the induced torsion free sheaf $\Fn$ is slope stable with respect to some ample divisor $D$ on $\Sn$. 
\end{thm}

\begin{proof}
By the results of Stapleton in \cite{staple} and in Section \ref{sec1} we know that for a given slope stable torsion free sheaf $E$ on $X$ with $E^{**}\neq\OO_X$, the associated tautological sheaf $\En$ is slope stable on $\Xn$. 

By Theorem \ref{desce} the sheaf $\En$ descends to $\Sn$ if and only if $E$ descends to $S$. In this case $\En\cong \rho^{*}\Fn$. Now by Theorem \ref{stability} the sheaf $\Fn$ is slope stable with respect to some $D\in \Amp(\Sn)$ if $\En$ is slope stable with respect to $H\in \Amp(\Xn)$ of the form $H=\rho^{*}D$ for some $D\in \Amp(\Sn)$.

To see that we find such a $D\in \Amp(\Sn)$, we note that the divisor $H$ is described quite explicitly in \cite[Proposition 4.8]{staple}: it is of the form 
\begin{equation*}
H=h_n+\epsilon A
\end{equation*}
for an \emph{arbitrary} ample divisor $A$ on $\Xn$ and $\epsilon$ sufficiently small. We choose $A$ of the form $A=\rho^{*}C$ for some $C\in \Amp(\Sn)$. By \eqref{iondiv} we also have
\begin{equation*}
\left( \In\right)^{*}\left( h_n\right)= (\iota^{*}h)_n  = \left(\iota^{*}\pi^{*}d \right)_n =  \left(\pi^{*}d \right)_n=h_n
\end{equation*}
which implies that we must have that $h_n=\rho^{*}B$ for some divisor $B$ on $\Sn$. Putting both facts together shows
\begin{equation*}
H=\rho^{*}D\,\,\,\text{for}\,\,\, D=B+\epsilon C.
\end{equation*}
It remains to see that $D$ is ample. But since $\rho$ is finite and surjective $D$ is ample if and only if $\rho^{*}D=H$ is ample, see \cite[Proposition I.4.4]{hart}.
\end{proof}

In the rest of this section we want to study the moduli spaces containing the slope stable sheaves $F$ on $S$ and $\Fn$ on $\Sn$. For this we let $v\in \mathrm{H}^{*}_{\mathrm{alg}}(S,\mathbb{Z})$ be a Mukai vector on $S$, that is $v=\operatorname{ch}(F)\sqrt{\operatorname{td}(S)}$ for some $F\in \Coh(S)$. Here
\begin{equation*}
\mathrm{H}^{*}_{\mathrm{alg}}(S,\mathbb{Z})=\mathrm{H}^0(S,\mathbb{Z})\oplus \mathrm{Num}(S)\oplus \frac{1}{2}\mathbb{Z}\xi_S.
\end{equation*}
where $\xi_S$ denotes the fundamental class of $S$.

We begin with the following result:
\begin{thm}\label{stab2}
Let $F$ be a torsion free coherent sheaf with $F\not\cong F\otimes\omega_S$. If $F$ is slope stable with respect to $d\in \Amp(S)$, $F^{**}\not\cong \OO_S$ and $F^{**}\not\cong \omega_S$, then $\Fn$ is a slope stable torsion free coherent sheaf on $\Sn$.
\end{thm}

\begin{proof}
The assumptions imply that $F$ is simple and that $\Hom_S(F,F\otimes\omega_S)=0$. Hence $E:=\pi^{*}F$ is is simple due to the formula
\begin{equation*}
\Hom_X(E,E)\cong \Hom_S(F,F)\oplus\Hom_S(F,F\otimes\omega_S).
\end{equation*}
By \cite[Lemma 3.2.3]{lehn}, the sheaf $E$ is polystable with respect to $h=\pi^{*}d$. Being simple and polystable, E is stable.

Since $E^{**}\not\cong \OO_X$ the sheaf $\En$ is slope stable with respect to some $H\in \Amp(\Xn)$ and descends to $\Sn$ via $\En\cong \rho^{*}\Fn$. Now Theorem \ref{Fnstb} implies that $\Fn$ is slope stable with respect to some $D\in \Amp(\Sn)$ satisfying $\rho^{*}D=H$.
\end{proof}

\begin{rem}
    Every torsion free coherent sheaf $F$ of odd rank satisfies the condition $F\not\cong F\otimes\omega_S$.
\end{rem}

Assume from now on, that $S$ is an unnodal Enriques surface, that is $S$ contains no smooth rational curves (that is no $(-2)$-curves). Note that in the moduli space of Enriques surfaces, a very general element will be unnodal by \cite[Corollary 5.7]{nami}.

Denote the moduli space of slope semistable sheaves (with respect to $d\in \Amp(S)$) with Mukai vector $v$ on $S$ by $\mathrm{M}_{S,d}(v)$. Assume that $v$ is primitive and chosen such that every slope semistable sheaf is slope stable and the rank of $v$ is odd. Then for a generic choice of $d\in \Amp(S)$ the moduli space $\mathrm{M}_{S,d}(v)$ is smooth of dimension $v^2+1$ and $\mathrm{M}_{S,d}(v)\neq \emptyset$ if and only if $v^2\geqslant -1$, see \cite[Proposition 4.2, Theorem 4.6 (i)]{yoshi}.

Furthermore in this situation there is a decomposition 
\begin{equation}\label{decomp}
\mathrm{M}_{S,d}(v)=\mathrm{M}_{S,d}(v,L_1)\coprod \mathrm{M}_{S,d}(v,L_2)
\end{equation}
where $\mathrm{M}_{S,d}(v,L_i)$ contains those $[E]\in \mathrm{M}_{S,d}(v)$ with $\det(E)=L_i$ where $L_2=L_1\otimes\omega_S$, that is $\operatorname{c}_1=\operatorname{c}_1(L_1)=\operatorname{c}_1(L_2)\in \mathrm{Num}(S)$. By \cite[Theorem 4.6.(ii)]{yoshi} for a general choice of $d\in \Amp(S)$ the moduli space $\mathrm{M}_{S,d}(v,L)$ is irreducible, that is a smooth projective variety.

We also assume that the Mukai vector is chosen such that for all $[F]\in \M_{S,d}(v,L)$ the sheaf $F$ is locally free on $S$ and does not have higher cohomology. Denote the Mukai vector of the associated sheaf $\Fn$ on $\Sn$ by $v_{[n]}$. If $\Fn$ is slope stable with respect to some $D\in \Amp(\Sn)$, denote its moduli space by $\mathcal{M}_{\Sn,D}(v_{[n]})$.

\begin{prop}\label{oned}
If $v\neq v(\OO_S)=v(\omega_S)$ then there is a class $D\in \Amp(\Sn)$ such that $\Fn$ is slope stable with respect to $D$ for all $[F]\in \M_{S,d}(v,L)$.
\end{prop}
\begin{proof}
Since all sheaves classified by $\mathrm{M}_{S,d}(v,L)$ are locally free on $S$, so are all the $E=\pi^{*}F$ on $X$. Proposition \ref{oneh} shows that there is one $H\in \Amp(\Xn)$ such that all $\En$ are slope stable with respect to $H$ since $E\not\cong\OO_X$. But then by the construction of $D\in \Amp(\Sn)$ with $H=\rho^{*}D$ in Theorem \ref{Fnstb}, it follows that there is one such desired $D$.
\end{proof}

We have the following corollary:

\begin{cor}
If $v\neq v(\OO_S)=v(\omega_S)$, then functor $(-)_{[n]}$ induces a morphism
\begin{equation*}
(-)_{[n]}: \mathrm{M}_{S,d}(v,L) \rightarrow \mathcal{M}_{\Sn,D}(v_{[n]}),\, [F]\mapsto [\Fn]
\end{equation*}
which identifies $\mathrm{M}_{S,d}(v,L)$ with a smooth connected component of $\mathcal{M}_{\Sn,D}(v_{[n]})$.
\end{cor} 

\begin{proof} 
We use the explicit description $(-)_{[n]}={p_{\Sn*}}(p_S^{*}(-))$ given by Theorem \ref{descr}.
Since $p_X$ and $p_{\Xn}$ are flat we know by faithfully flat descent for $\pi$ resp. $\rho$ that the induced projections $p_S$ and $p_{\Sn}$ are flat. Similarly since $p_{\Xn}$ is a finite morphism so is $p_{\Sn}$.

Using these facts together with Theorem \ref{stab2} and Proposition \ref{oned} shows that Krug's argument in the proof of \cite[Proposition 2.1]{krug5} also works in this case. Hence $[F]\mapsto [F_{[n]}]$ is a regular morphism. 

Similar to Theorem \ref{compo} it follows from Theorem \ref{homext} that $(-)_{[n]}$ is injective on closed points as $\Hom_{\Sn}(\Fn,\Gn)\cong \Hom_S(F,G)$. By Theorem \ref{homext} we also have 
\begin{equation*}
\dim(\Ext^1_{\Sn}(\Fn,\Fn))=\dim(\Ext_S^1(F,F)).
\end{equation*}
Both facts together imply that $(-)_{[n]}$ identifies $\mathrm{M}_{S,d}(v,L)$ with a smooth connected component of $\mathcal{M}_{\Sn,D}(v_{[n]})$.
\end{proof}

\begin{rem}
    There is a decomposition
    \begin{equation*}
        \mathcal{M}_{\Sn,D}(v_{[n]})=\mathcal{M}_{\Sn,D}(v_{[n]},\mathcal{L}_1)\coprod \mathcal{M}_{\Sn,D}(v_{[n]},\mathcal{L}_2)
    \end{equation*}
    analogous to \eqref{decomp} and, depending on the choice of $(-)_{[n]}$ (see Remark \ref{choices}), $\M_{S,d}(v,L)$ is mapped  to a component of $\mathcal{M}_{\Sn,D}(v_{[n]},\mathcal{L}_1)$ or a component of $\mathcal{M}_{\Sn,D}(v_{[n]},\mathcal{L}_2)$.
\end{rem}

Denote the Mukai vector of $E=\pi^{*}F$ on $X$ by $w$, that is $w=\pi^{*}v$. In the rest of this section we want to study the fixed loci of $\iota^{*}$ in $\mathrm{M}_{X,h}(w)$ and $\left(\In\right)^{*}$ in $\mathcal{M}_{\Xn,H}(w^{[n]})$. In our situation we have a well defined morphism
\begin{equation*}
\pi^{*}: \mathrm{M}_{S,d}(v) \rightarrow \mathrm{M}_{X,h}(w),\,\,\,F\mapsto \pi^{*}F
\end{equation*} 
which has image in $\Fix(\iota^{*})$. More exactly the image of $\pi^{*}$ is the fixed locus of $\iota^{*}$ and the morphism restricts to an \'{e}tale 2:1-morphism
\begin{equation*}
\pi^{*}: \mathrm{M}_{S,d}(v)\rightarrow \Fix(\iota^{*}).
\end{equation*}
Furthermore $\Fix(\iota^{*})$ is a Lagrangian subscheme in $\mathrm{M}_{X,h}(w)$, see for example \cite[Theorem (1)]{kim} or \cite[Theorem 2.3 (c)]{nuer}.

As the morphism $\pi^{*}: \mathrm{M}_{S,d}(v) \rightarrow \Fix(\iota^{*})$ is an \'{e}tale 2:1-morphism, the decomposition \eqref{decomp} shows that $\pi^{*}$ induces an isomorphism $\mathrm{M}_{S,d}(v,L) \cong \Fix(\iota^{*})$. As $\mathrm{M}_{S,d}(v,L)$ is irreducible, so is $\Fix(\iota^{*})$.

\begin{thm}
The fixed locus $\Fix(\iota^{*})$ is a smooth projective variety. The morphism $(-)^{[n]}$ in Theorem \ref{compo} restricts to a morphism
\begin{equation*}
(-)^{[n]}: \Fix(\iota^{*}) \rightarrow \Fix(\left( \In\right) ^{*})
\end{equation*}
which identifies $\Fix(\iota^{*})$ with a smooth connected component of $\Fix(\left( \In\right) ^{*})$.
\end{thm}

\begin{proof}
The fixed locus $\Fix(\iota^{*})$ is smooth and projective since $\mathrm{M}_{X,h}(w)$ is smooth and projective. Since it is also irreducible, it is a smooth projective variety. 

By Lemma \ref{commute} the morphism $(-)^{[n]}$ restricts to a morphism between the fixed loci. Since $(-)^{[n]}$ is injective on closed points, so is its restriction to $\Fix(\iota^{*})$. 

To identify $\Fix(\iota^{*})$ as a smooth connected component it is therefore enough to prove
\begin{equation*}
\dim\left( T_{[E]}\Fix(\iota^{*}) \right)=\dim \left(T_{\left[ \En\right] }\Fix(\left(\In \right)^{*} ) \right)  
\end{equation*}
But a general fact says that the tangent space of the fixed locus satisfies
\begin{equation*}
T_y \left( Y^G \right) \cong (T_y Y)^G,
\end{equation*}
see for example \cite[Proposition 3.2]{edix}. As we have $E\cong\pi^{*}F$ for some sheaf $F$ on $S$, this shows
\begin{align*}
T_{[E]}\Fix(\iota^{*})\cong\left(T_{[E]}\mathrm{M}_{X,h}(w) \right)^{\iota}\cong \left( \Ext_X^1(E,E)\right) ^{\iota}\cong \Ext^1_S(F,F).
\end{align*}

A similar computation shows
\begin{equation*}
T_{[\En]}\Fix(\left(\In \right) ^{*})\cong\left( \Ext_{\Xn}^1(\En,\En)\right) ^{\In} \cong \Ext^1_{\Sn}(\Fn,\Fn)\cong \Ext^1_S(F,F)
\end{equation*}
by Theorem \ref{homext} since $\En\cong \rho^{*}\Fn$.
\end{proof}

\begin{cor}
The diagram \eqref{diag} induces the commutative diagram:
 	\begin{equation*}
 		\begin{tikzcd}
 			\Fix(\iota^{*}) \arrow[r,"(-)^{[n]}"] &[1cm]\Fix(\left(\In \right)^{*}) \\[.5cm]
 			\mathrm{M}_{S,d}(v,L) \arrow[r,swap,"(-)_{[n]}"]\arrow[u,"\pi^{*}"]  & \mathcal{M}_{\Sn,D}(v_{[n]})\arrow[u,swap,"\rho^{*}"]
 		\end{tikzcd}.
 	\end{equation*}	
\end{cor}

\section{Acknowledgement}
I thank Andreas Krug for answering my (many) questions regarding \cite{krug4}. I also thank the referee for a very detailed report which improved the paper greatly.

\end{document}